\documentclass[11pt,reqno]{amsart}
\usepackage{amsmath}
\usepackage{cases}
\usepackage{mathrsfs}
\usepackage{bbm}
\usepackage{amssymb}
\usepackage{amscd}
\usepackage{amsfonts,latexsym,amsmath,
amsthm,amsxtra,mathdots,amssymb,latexsym,mathabx}
\usepackage[all,cmtip]{xy}
\RequirePackage{amsmath} \RequirePackage{amssymb}
\usepackage{color}
\usepackage{colordvi}
\usepackage{multicol}
\usepackage{hyperref}
\usepackage{mathtools}
\usepackage[margin=1in]{geometry}
\usepackage{xcolor}

\hypersetup{
    colorlinks,
    linkcolor={red!50!black},
    citecolor={blue!100!black},
    urlcolor={blue!100!black}
}
\usepackage{cite}

\newcommand{\bea}{\begin{eqnarray}}
\newcommand{\eea}{\end{eqnarray}}
\newcommand{\bna}{\begin{eqnarray*}}
\newcommand{\ena}{\end{eqnarray*}}

\numberwithin{equation}{section}
\DeclareMathOperator\arccosh{arccosh}

\theoremstyle{plain}
\newtheorem{lemma}{Lemma}[section]
\newtheorem{theorem}[lemma]{Theorem}

\theoremstyle{definition}

\newtheorem{remark}{Remark}

\renewcommand{\Re}{\operatorname{Re}}
\renewcommand{\Im}{\operatorname{Im}}

\newcommand{\GL}{\operatorname{GL}}

\title[]
{On an unconditional spectral analog of Selberg's result on $S(t)$}

\author{Qingfeng Sun}
	\address{School of Mathematics and Statistics, Shandong University, Weihai\\Weihai, Shandong 264209, China}
	\email{qfsun@sdu.edu.cn}

    \author{Hui Wang}
  \address{School of Mathematics, Shandong University, Jinan 250100, China}
    \email{wh0315@mail.sdu.edu.cn}



\thanks{Q. Sun was partially supported by the National Natural Science Foundation of China (Grant Nos.
11871306 and 12031008) and
the Natural Science Foundation of Shandong Province (Grant No. ZR2023MA003).}

\date{}

\begin{document}
\begin{abstract}
Let $S_j(t)=\frac{1}{\pi}\arg L(1/2+it, u_j)$, where $u_j$ is an even Hecke--Maass cusp form
for $\rm SL_2(\mathbb{Z})$ with Laplacian eigenvalue $\lambda_j=\frac{1}{4}+t_j^2$.
Without assuming the GRH, we establish an asymptotic formula for the moments of $S_j(t)$.
\end{abstract}
\maketitle

\section{Introduction}
\setcounter{equation}{0}
For $T>0$, the approximate formula for $N(T)$,
the number of zeros $\rho=\beta+i\gamma$ of the Riemann zeta-function $\zeta(s)$
in the rectangle $0<\beta<1$, $0<\gamma\leq T$,
was given by \cite[Chapter 15]{Davenport} of the form
\bna
N(T)=\frac{T}{2\pi}\log \frac{T}{2\pi}-\frac{T}{2\pi}+\frac{7}{8}+S(T)+O(T^{-1}),
\ena
where \bea\label{S(T)}
S(T)=\frac{1}{\pi}\arg \zeta\big(\frac{1}{2}+iT\big).
\eea
The argument is defined by continuous variation along the horizontal
segment from $\infty+iT$ to $1/2+iT$, starting at $\infty+iT$ with the value $0$.
The complicated behavior of $S(t)$ is an interesting topic studied by many number theorists
(see, for example, \cite{Backlund,Cramer,Littlewood,Titchmarsh1}).

In 1940s, Selberg \cite{Selberg, Selberg1} showed that for any $n\in \mathbb{N}$,
\bea\label{S(t) even moment}
\int_{T}^{2T}S(t)^{2n}\mathrm{d}t=\frac{(2n)!}{n!(2\pi)^{2n}}T(\log\log T)^n+O\big(T(\log\log T)^{n-1/2}\big),
\eea
which can imply, on average sense, $|S(t)|$ has order of magnitude $\sqrt{\log\log T}$.
In addition, for Dirichlet $L$-functions $L(s,\chi)$ with primitive $\chi$ and
$$S(t,\chi)=\frac{1}{\pi}\arg L\big(\frac{1}{2}+it, \chi\big),$$
Selberg \cite{Selberg2} investigated the behavior of $S(t,\chi)$ and obtained, for prime modulus $q$ and fixed $t>0$,
\bea\label{S(chi) even moment}
\sideset{}{^*}\sum_{\chi \bmod q} S(t, \chi)^{2n}=\frac{(2n)!}{n!(2\pi)^{2n}}q(\log\log q)^n+O\big(q(\log\log q)^{n-1/2}\big),
\eea
where $*$ denotes the summation goes through the primitive characters $\chi (\bmod\,q)$.
Note that the counterparts of \eqref{S(t) even moment} and \eqref{S(chi) even moment} for the odd-order moments
can also be obtained by making a minor adaptation in Selberg's techniques.

For higher rank $L$-functions, there are relevant investigations.
In 1997, Hejhal--Luo \cite{HL} considered the generalization of \eqref{S(chi) even moment}
to the framework of $\Gamma$-automorphic
eigenfunctions $\psi$ of the hyperbolic Laplacian $\Delta$ for the full modular group $\Gamma=\rm SL_2(\mathbb{Z})$.
Let $L(s,\psi)$ be the usual Hecke--Maass $L$-function associated to $\psi$, and define
\bna
S(t,\psi)=\frac{1}{\pi}\arg L\big(\frac{1}{2}+it, \psi\big)
\ena
in the same fashion as in \eqref{S(T)}.
They established the analogs of \eqref{S(t) even moment} when $\psi$ is Eisenstain series.
And, let $\{u_j\}_{j=1}^\infty$ be an orthonormal basis of the space of
Maass cusp forms for $\rm SL_2(\mathbb{Z})$ such that
$\Delta u_j=\lambda_ju_j$ with Laplacian eigenvalue $\lambda_j=\frac{1}{4}+t_j^2$ ($t_j\geq0$),
and each $u_j$ is an eigenfunction of all the Hecke operators.
In this case $\{u_j\}_{j=1}^\infty$ consists of even Maass forms and odd Maass forms according to $u_j(-\overline{z})=u_j(z)$
and $u_j(-\overline{z})=-u_j(z)$.
Each $u_j(z)$ has the Fourier expansion
\bna
u_j(z)=\sqrt{y}\sum_{n\neq 0}^\infty \rho_j(n)K_{it_j}(2\pi |n|y)e(nx), \quad z=x+iy\in \mathbb{H},
\ena
where $K_\alpha(x)$ is the $K$-Bessel function of order $\alpha$ and
\bna
\rho_j(\pm n)=\rho_j(\pm1)\lambda_j(n), \quad n\geq 1.
\ena
Here $\rho_j(1)\neq 0$, and $\lambda_j(n)$ is the eigenvalue of the $n$-th Hecke operator $T_n$.
If $\psi$ is the Maass cusp form $u_j$, under Generalized Riemann Hypothesis (GRH),
they showed that for fixed $t, \beta>0$ and any large parameter $T>0$,
\begin{equation}\label{HL result}
\begin{split}
\sum_{j=1}^\infty e^{-\beta(t_j-T)^2}|\nu_j(1)|^2(S(t,u_j))^{2n}=
\frac{2\pi^{-3/2}}{\sqrt{\beta}}\frac{(2n)!}{n!(2\pi)^{2n}}T(\log\log T)^{n}
+O_{t,n,\beta}\big(T(\log\log T)^{n-1/2}\big),
\end{split}
\end{equation}
where $\nu_j(1)$ is defined by \eqref{nu_j(n)} below.
These spectral results have been recently extended, with a $\rm\GL_3$ analog obtained by Liu--Liu \cite{LL} under GRH,
which was generalized by Chen--Lau--Wang \cite[Corollary 1.2]{CLW} to $\rm\GL_n$, $n\geq 3$ under the same hypothesis.
Moreover, for the holomorphic case, by means of the weighted version of the zero-density estimate,
Liu--Shim \cite{LS} and Sun--Wang \cite{SW} successively obtained the unconditional analogues in the weight aspect and in the level aspect, respectively.

In this paper, we are devoted to get the asymptotic formula \eqref{HL result} after removing the GRH.
Assume that $u_j$ is an even Hecke--Maass cusp form as above, the $L$-series associated to it is defined by
\begin{equation*}\label{L-functions}
\begin{split}
L(s,u_j)=\sum_{n=1}^\infty\frac{\lambda_j(n)}{n^s}&=\prod_p(1-\lambda_j(p)p^{-s}+p^{-2s})^{-1}\\
&=\prod_p(1-\alpha_j(p)p^{-s})^{-1}(1-\beta_j(p)p^{-s})^{-1},\quad \text{for}\,\,\Re(s)>1,
\end{split}
\end{equation*}
which admits an analytic continuation to the whole complex plane and satisfies the functional equation
\bna\label{FE1}
\Lambda(s,u_j)=\pi^{-s}\,
\Gamma\left(\frac{s+it_j}{2}\right)\Gamma\left(\frac{s-it_j}{2}\right)L(s,u_j)=\Lambda(1-s,u_j).
\ena
Here we have $\alpha_j(p)\beta_j(p)=1$ and, for convenience, $|\alpha_j(p)|\geq 1$.
It is not known that the $L$-functions of Maass cusp forms satisfy the Ramanujan--Petersson conjecture, the
current best known estimate is due to Kim--Sarnak \cite{KS},
\bea\label{lambda_j(n) bound}
|\lambda_j(n)|\leq n^\vartheta\tau(n), \quad n\in\mathbb{N},
\eea
with $\vartheta=7/64$ and $\tau(n)$ being the divisor function. The Ramanujan--Petersson conjecture predicts $\vartheta=0$.
Note that the Fourier coefficients $\lambda_j(n)$ are real and satisfy
the following Hecke relation,
\bea\label{the Hecke relation 1}
\lambda_j(m)\lambda_j(n)
=\sum\limits_{d|(m,n)}\lambda_j\left(\frac{mn}{d^2}\right),
\quad \text{for}\,\, m,n\geq 1.
\eea

Now we redefine $S_j(t)$ by
\bna
S_j(t):=\frac{1}{\pi}\arg L(1/2+it, u_j),
\ena
where the argument $\arg L(1/2+it, u_j)$ is obtained by continuous variation along the line $\Im(s)=t$ from
$\sigma=+\infty$ to $\sigma=1/2$.
Our main result is the following theorem.
\begin{theorem}\label{main-theorem}
Let $t>0$ be fixed and $T$ be any large parameter.
For $M$ a parameter with the constraint $T^\varepsilon\leq M\leq T^{1-\varepsilon}$, we have
\begin{equation*}
\begin{split}
\sideset{}{^\prime}\sum_{j=1}^\infty e^{-\frac{(t_j-T)^2}{M^2}}|\nu_j(1)|^2 (S_j(t))^n=
\frac{MT}{4\pi^{3/2}}C_n(\log\log T)^{n/2}
+O_{t,n}\big(MT(\log\log T)^{(n-1)/2}\big),
\end{split}
\end{equation*}
where $\nu_j(n)$ is defined by \eqref{nu_j(n)} below and
\bna
C_n=\begin{cases}
\frac{n!}{(n/2)!(2\pi)^{n}}, \,&\textit{if }\, n\,\, \text{is even},\\
0, \,& \textit{if }\, n \,\,\text{is odd}.
\end{cases}
\ena

Besides, we define the following probability measure
$\mu_{TM}$ on $\mathbb{R}$ by
\bna
\mu_{TM}(E):=\left(\sideset{}{^\prime}\sum_{|t_j-T|\leq M}|\nu_j(1)|^2
\textbf{1}_E\left(\frac{S_j(t)}{\sqrt{\log\log T}}\right)\right)
\bigg/\left(\sideset{}{^\prime}\sum_{|t_j-T|\leq M}|\nu_j(1)|^2\right),
\ena
where $\textbf{1}_E$ is the characteristic function on a Borel measurable set $E$ in $\mathbb{R}$.
Then for every $a<b$, we have
\bna
\lim_{T\rightarrow \infty} \mu_{TM}([a,b])=\sqrt{\pi}\int_a^b \exp(-\pi^2\xi^2)\mathrm{d}\xi,
\ena
uniformly with respect to $M$.
\end{theorem}
\begin{remark}
For $u_j$ an even Hecke--Maass cusp form of $\rm SL_2(\mathbb{Z})$
with Laplacian eigenvalue $\lambda_j=\frac{1}{4}+t_j^2$,
our result indicates that $|S_j(t)|$ has order of magnitude $\sqrt{\log\log T}$ on average
for $t_j\asymp T$. In addition, as $T\rightarrow \infty$, the ratios $S_j(t)/\sqrt{\log\log T}$ converges in distribution
to a normal distribution of mean $0$
and variance $(2\pi^2)^{-1}$.
\end{remark}
\section{Preliminaries}
\label{prelim} \subsection{Notation} Throughout the paper,
$\varepsilon$ is an arbitrarily small positive
real number, which may be different at each occurrence.
As usual, $e(z) = \exp (2 \pi i z) = e^{2 \pi i z}$ and $\delta_{m,n}$ is the Kronecker symbol that detects $m=n$.
We write $f = O(g)$ or $f \ll g$ to mean $|f| \leq cg$ for some unspecified positive constant $c$.
The symbol $y\asymp Y$ is to mean that $c_1Y\leq |y| \leq c_2Y$ for some positive constants $c_1$ and $c_2$.

The Kloosterman sum is defined by
\bna
S(m,n;c)=\sideset{}{^\star}\sum_{a\bmod c}e\left(\frac{ma+n\overline{a}}{c}\right),
\ena
where $\sideset{}{^\star}\sum$ restricts the summation to the primitive residue classes,
and $\overline{a}$ denotes the multiplicative inverse of $a$ modulo $c$.
We recall here Weil's bound in \cite{Weil} for the Kloosterman sum
\bna
|S(m,n;c)|\leq (m,n,c)^{1/2}c^{1/2}\tau(c).
\ena

Let $J_\nu(z)$, $I_\nu(z)$ and $K_\nu(z)$ be the Bessel functions of the first kind,
modified Bessel function of the first kind and
modified Bessel function of the second kind, respectively (see Watson \cite{Watson}).
Recall that $J_\nu(z)$ and $I_\nu(z)$ are entire in $\nu$ for fixed $z\neq 0$, and for $\Re(\nu)>-\frac{1}{2}$,
one has the integral representations \cite[pg. 912, Eq. 8.411(4) \& pg. 916, Eq. 8.431(3)]{GR},
\bna
J_\nu(z)=\frac{2(z/2)^\nu}{\Gamma(\frac{1}{2}+\nu)\Gamma(\frac{1}{2})}
\int_0^{\frac{\pi}{2}}\sin^{2\nu}\theta\cos(z\cos\theta)\mathrm{d}\theta,
\ena
and
\bna
I_\nu(z)=\frac{(z/2)^\nu}{\Gamma(\frac{1}{2}+\nu)\Gamma(\frac{1}{2})}\int_0^\pi
e^{\pm z\cos \theta}\sin^{2\nu} \theta\mathrm{d}\theta.
\ena
By Stirling's formula (see \cite[\S 8.4, Eq. (4.03)]{O}), for $x>0$,
we have
\bea\label{J-Bessel}
J_{\sigma+it}(x)\ll_\sigma
\left\{
\begin{aligned}
&x^\sigma, \quad && \text{if}\,\,|t|\leq 1,\\
&x^\sigma|t|^{-\sigma}e^{\pi|t|/2}, \quad &&\text{if}\,\,|t|\geq1,
\end{aligned}
\right.
\eea
and
\bea\label{I-Bessel}
I_{\sigma+it}(x)\ll_\sigma
\left\{
\begin{aligned}
&x^\sigma e^x, \quad && \text{if}\,\,|t|\leq 1,\\
&x^\sigma e^x|t|^{-\sigma}e^{\pi|t|/2}, \quad &&\text{if}\,\,|t|\geq1.
\end{aligned}
\right.
\eea
Moreover, we have the identity \cite[pg. 928, Eq. 8.485]{GR}
\bea\label{K-Bessel}
K_\nu(z)=\frac{\pi}{2}\frac{I_{-\nu}(z)-I_{\nu}(z)}{\sin (\pi \nu)}, \quad\text{for}\,\,\nu\,\, \text{is not an integer},
\eea
and the bound (e.g., \cite[Corollary 3.2]{GRS})
\bea\label{K-Bessel bound}
K_{2it}(x)\ll e^{-\pi t}t^{-\frac{1}{3}},\quad \text{for all}\,\, t>0 \,\,\text{and}\,\, x\geq 1.
\eea

Throughout the paper, we define
\bea\label{nu_j(n)}
\nu_j(n):=\rho_j(n)/\sqrt{\cosh(\pi t_j)},
\eea
and immediately obtain, $\nu_j(n)=\nu_j(1)\lambda_j(n).$
\subsection{Kuznetsov trace formula over even Hecke--Maass forms}
\begin{lemma}\label{Kuznetsov trace formula}
Let $h(t)$ be an even function which is holomorphic in the strip $|\Im(t)|\leq\frac{1}{2}+\varepsilon$ and
$h(t)\ll (1+|t|)^{-2-\varepsilon}$ with some $\varepsilon>0$. Then for integers $m,n\geq 1$,
\begin{equation*}
\begin{split}
&\sideset{}{^\prime}\sum_{j}h(t_j)|\nu_j(1)|^2\lambda_j(m)\lambda_j(n)
+\frac{1}{4\pi} \int_{-\infty}^\infty \frac{\tau_{it}(m)\tau_{it}(n)}{|\zeta(1+2it)|^2}h(t)\mathrm{d}t\\
=&\delta_{m,n}\mathcal{H}+\sum_{c>0}
\frac{1}{2c}\left(S(m,n;c)\mathcal{H}^+\left(\frac{4\pi\sqrt{mn}}{c}\right)+S(-m,n;c)\mathcal{H}^-\left(\frac{4\pi\sqrt{mn}}{c}\right)\right).
\end{split}
\end{equation*}
Here $\sideset{}{^\prime}\sum$ restricts the sum to even Hecke--Maass cusp forms,
$\tau_\nu(n)=\sum_{ab=n}\left(\frac{a}{b}\right)^\nu$ is the generalized divisor function,
\begin{equation*}
\begin{split}
\mathcal{H}=\frac{1}{8\pi^2}\int_{-\infty}^\infty h(t)\tanh(\pi t)t\, \mathrm{d}t,\\
\end{split}
\end{equation*}
and
\begin{equation*}
\begin{split}
\mathcal{H}^+(x)&=\frac{i}{2\pi}\int_{-\infty}^\infty J_{2it}(x)\frac{h(t)t}{\cosh(\pi t)}\,\mathrm{d}t,\\
\mathcal{H}^-(x)&=\frac{1}{\pi^2}\int_{-\infty}^\infty K_{2it}(x)\sinh(\pi t)h(t)t\mathrm{d}t.
\end{split}
\end{equation*}
\end{lemma}
\begin{proof}
See e.g., Conrey--Iwaniec \cite{CI}.
\end{proof}
By \cite{HL2, Iwaniec} and the standard bounds of the Riemann zeta-function (e.g., \cite[\S 3.11]{Titchmarsh}), we have
\bea\label{bounds}
(\log t_j)^{-1}\ll |\nu_j(1)|^2\ll \log t_j,
\qquad
(\log(1+|t|))^{-1}\ll \zeta(1+2it)\ll \log(1+|t|).
\eea
According to \cite[Theorem 6]{Kuznetsov} and \cite{Hejhal}, respectively, we have
\begin{equation}\label{classic bounds}
\begin{split}
\sum_{t_j\leq T}|\nu_j(1)|^2=\frac{T^2}{\pi^2}+O(T\log T),
\end{split}
\end{equation}
and for some constants $\beta_1$, $\beta_2$,
\begin{equation*}\label{classic bound1}
\begin{split}
\sum_{t_j\leq T}1&=\frac{T^2}{12}+\beta_1 T\log T+\beta_2 T+O\bigg(\frac{T}{\log T}\bigg).
\end{split}
\end{equation*}
\begin{remark}
Note that Li \cite[Corollary 1.2]{Li} proved the currently best error term of the (weighted)
Weyl law in \eqref{classic bounds} is $O(T)$.
\end{remark}

As a corollary of Lemma \ref{Kuznetsov trace formula},
we give the following lemma and postpone its proof to Section \ref{Prove Lemma}.
\begin{lemma}\label{corollary1}
Let $T\geq 3$ and large parameter $M$ with $T^\varepsilon\leq M\leq T^{1-\varepsilon}$.
For any integers $m,n\geq 1$ satisfying $\max\{m,n\}\leq T^{1-\varepsilon}$ and any $\varepsilon$ with $0<\varepsilon<1/2$, we have
\bna
\sideset{}{^\prime}\sum_{j} e^{-\frac{(t_j-T)^2}{M^2}}|\nu_j(1)|^2\lambda_j(m)\lambda_j(n)=\frac{1}{4\pi^{3/2}}MT\,\delta_{m,n}
+O_{\varepsilon}\big(MT^\varepsilon\big).
\ena
\end{lemma}
\subsection{The zero-density estimate}
We will need the following weighted version of zero-density estimate
in the spectral aspect for the $L$-functions associated to even Hecke--Maass forms
which was established by Liu--Streipel \cite[Theorem 1.3]{LS24}.
\begin{lemma}\label{zero-density spectral}
For $\sigma>\frac{1}{2}$ and $H>0$,
we let $$N(\sigma, H; u_j):=\#\{\rho=\beta+i\gamma \,|\, L(\rho,u_j)=0, \sigma<\beta, |\gamma|<H\},$$
and
\bna
N(\sigma, H):=\frac{1}{\mathcal{H}}\sideset{}{^\prime}\sum_j |\nu_j(1)|^2 h(t_j)N(\sigma, H; u_j),
\ena
where $h(t)$ is defined by
\bea\label{definition h(t)}
h(t)=e^{-\frac{(t-T)^2}{M^2}}+e^{-\frac{(t+T)^2}{M^2}},\quad\text{for}\,\, 1<M<T.
\eea
Let $\frac{1}{2}+\frac{2}{\log T}<\sigma< 1$. Thus for some sufficiently small $\delta_1,\theta>0$, we have
\bna
N(\sigma, H)\ll HT^{-\theta(\sigma-1/2)}\log T
\ena
uniformly in $3/\log T<H<T^{\delta_1}$.
\end{lemma}

\section{An approximation of $S_j(t)$}
In this section we will prove several lemmas and derive an approximation of $S_j(t)$.
We denote by $\rho=\beta+i\gamma$ a typical zero of $L(s,u_j)$ inside the critical strip, i.e., $0<\beta<1$.
For $\Re(s)>1$, we have the Dirichlet series expansion for the logarithmic derivative of $L(s,u_j)$,
\bea\label{Log derivative definition spectral}
-\frac{L'}{L}(s,u_j)=\sum_{n=1}^\infty \frac{\Lambda(n)C_j(n)}{n^s},
\eea
where $\Lambda(n)$ denotes the von Mangoldt function, and
\bea\label{Log derivative coefficient spectral}
C_j(n)=\begin{cases}
\alpha_j(p)^m+\beta_j(p)^m, \,&\text{if}\,\,n=p^m\,\,\text{for a prime}\,\,p,\\
0, \,& \text{otherwise}.
\end{cases}
\eea
Using the facts that
\bna
C_j(p^m)=\alpha_j(p)^m+\alpha_j(p)^{-m},
\qquad
\text{and}
\qquad
\lambda_j(p^\nu)=\sum_{j=0}^\nu\alpha_j(p)^{\nu-2j},
\ena
we immediately check that
\bea\label{C_j prime power}
C_j(p)=\lambda_j(p), \quad \text{and}\quad C_j(p^r)=\lambda_j(p^r)-\lambda_j(p^{r-2}),\,\,\text{for}\,\,r\geq 2.
\eea
\begin{lemma}\label{lemma 1 spectral}
Let $x>1$. For $s\neq \rho$, and $s\neq \pm it_j-2m$, we have the following identity
\begin{equation*}
\begin{split}
\frac{L'}{L}(s,u_j)=&-\sum_{n\leq x^3} \frac{C_j(n)\Lambda_x(n)}{n^s}
-\frac{1}{\log^2x}\sum_\rho\frac{x^{\rho-s}(1-x^{\rho-s})^2}{(\rho-s)^3}\\
&-\frac{1}{\log^2x}\sum_{m=0}^\infty\frac{x^{\pm it_j-2m-s}(1-x^{\pm it_j-2m-s})^2}{(\pm it_j-2m-s)^3},
\end{split}
\end{equation*}
where
\bna
\Lambda_x(n)=\begin{cases}
\Lambda(n), \,&\textit{if }\,\,n\leq x,\\
\Lambda(n)\frac{\log^2(x^3/n)-2\log^2(x^2/n)}{2\log^2x}, \,& \textit{if }\,\,x\leq n\leq x^2,\\
\Lambda(n)\frac{\log^2(x^3/n)}{2\log^2x}, \,& \textit{if }\,\,x^2\leq n\leq x^3,\\
0, \,& \textit{if }\,\,n\geq x^3.\\
\end{cases}
\ena
\end{lemma}
\begin{proof}
Recall the discontinuous integral
\bna
\frac{1}{2\pi i}\int_{(\alpha)}\frac{y^s}{s^3}\mathrm{d}s=
\begin{cases}
\frac{\log^2 y}{2}, \,&\textit{if }\,\, y\geq 1,\\
0, \,& \textit{if }\,\, 0<y\leq 1
\end{cases}
\ena
for $\alpha>0$.
It follows from \eqref{Log derivative definition spectral} and \eqref{Log derivative coefficient spectral} that
\bna
-\log^2x\sum_{n=1}^\infty \frac{C_j(n)\Lambda_x(n)}{n^s}=
\frac{1}{2\pi i}\int_{(\alpha)}\frac{x^v(1-x^v)^2}{v^3}\frac{L'}{L}(s+v,u_j)\mathrm{d}v,
\ena
where $\alpha=\max \{2, 1+\Re(s)\}$. By moving the line of integration all way to the left, we
pick up the residues at $v=0$, $v=\rho-s$ and $v=\pm it_j-2m-s$ ($m= 0,1,2, \cdots$), and deduce that
\begin{equation*}
\begin{split}
&\frac{1}{2\pi i}\int_{(\alpha)}\frac{x^v(1-x^v)^2}{v^3}\frac{L'}{L}(s+v,u_j)\mathrm{d}v\\
=&\frac{L'}{L}(s,u_j)\log^2x+\sum_\rho\frac{x^{\rho-s}(1-x^{\rho-s})^2}{(\rho-s)^3}
+\sum_{m=0}^\infty\frac{x^{\pm it_j-2m-s}(1-x^{\pm it_j-2m-s})^2}{(\pm it_j-2m-s)^3}.
\end{split}
\end{equation*}
Thus we complete the proof of lemma.
\end{proof}
\begin{lemma}\label{lemma 2 spectral}
For $s=\sigma+it$, $s'=\sigma'+it'$ such that $1/2\leq \sigma, \sigma'\leq 10$, $s\neq \rho$, $s'\neq \rho$,
we have
\bna\label{Im spectral}
\Im\left(\frac{L'}{L}(s,u_j)-\frac{L'}{L}(s',u_j)\right)=\Im\sum_\rho\left(\frac{1}{s-\rho}-\frac{1}{s'-\rho}\right)+O(1),
\ena
and
\bna\label{Re spectral}
\Re\frac{L'}{L}(s,u_j)=\sum_\rho\frac{\sigma-\beta}{(\sigma-\beta)^2+(t-\gamma)^2}+O(\log (|t|+t_j+1)).
\ena
\end{lemma}
\begin{proof}
By the Hadamard factorization theorem of the entire function $\Lambda(s,u_j)$, we have
\bna
\frac{L'}{L}(s,u_j)=b_j+\sum_\rho\left(\frac{1}{s-\rho}+\frac{1}{\rho}\right)
-\frac{1}{2}\frac{\Gamma'}{\Gamma}\left(\frac{s+it_j}{2}\right)
-\frac{1}{2}\frac{\Gamma'}{\Gamma}\left(\frac{s-it_j}{2}\right)+\log\pi,
\ena
for some $b_j\in \mathbb{C}$ with $\Re(b_j)=-\Re\sum_\rho\frac{1}{\rho}$ (see \cite[Proposition 5.7 (3)]{IK}).
Note that for $z\not\in -\mathbb{N}$,
\bea\label{log derivative gamma spectral}
\frac{\Gamma'}{\Gamma}(z)=-\gamma+\sum_{k\geq 1}\left(\frac{1}{k}-\frac{1}{k-1+z}\right),
\eea
where $\gamma$ is the Euler constant.
Using \eqref{log derivative gamma spectral}, we get that for $\frac{1}{4}\leq u$,
\bna
\Im\frac{\Gamma'}{\Gamma}(u+iv)=\sum_{k\geq 1}\frac{v}{(k+u-1)^2+v^2}\ll 1.
\ena
Combined with the asymptotic formula of the logarithmic derivative of $\Gamma(z)$,
\bna\label{derivative gamma spectral}
\frac{\Gamma'}{\Gamma}(z)=\log z-\frac{1}{2z}+O_\varepsilon\left(\frac{1}{|z|^2}\right),
\quad\text{for}\,\, |\arg z|\leq \pi-\varepsilon,
\ena
we complete the proof of lemma.
\end{proof}
Let $x\geq 4$. We define
\bea\label{sigma_x spectral}
\sigma_x=\sigma_{x,j}=\sigma_{x,j,t}:=\frac{1}{2}+2\max_\rho\left\{\left|\beta-\frac{1}{2}\right|, \frac{5}{\log x}\right\},
\eea
where $\rho=\beta+i\gamma$ runs through the zeros of $L(s,u_j)$ for which
\bna\label{condition 1 spectral}
|t-\gamma|\leq\frac{x^{3|\beta-1/2|}}{\log x}.
\ena

The following lemma is to display the dependence of $\sigma_x$ as needed.
\begin{lemma}\label{lemma 3 spectral}
Let $x\geq 4$. For $\sigma\geq \sigma_x$, we have
\begin{equation*}
\begin{split}
\frac{L'}{L}(\sigma+it,u_j)=&-\sum_{n\leq x^3} \frac{C_j(n)\Lambda_x(n)}{n^{\sigma+it}}
+O\left(x^{1/4-\sigma/2}\left|\sum_{n\leq x^3} \frac{C_j(n)\Lambda_x(n)}{n^{\sigma_x+it}}\right|\right)\\
&+O\left(x^{1/4-\sigma/2}\log(|t|+t_j+1)\right),
\end{split}
\end{equation*}
and
\bna
\sum_\rho\frac{\sigma_x-1/2}{(\sigma_x-\beta)^2+(t-\gamma)^2}=
O\left(\left|\sum_{n\leq x^3} \frac{C_j(n)\Lambda_x(n)}{n^{\sigma_x+it}}\right|\right)
+O\big(\log(|t|+t_j+1)\big).
\ena
\end{lemma}
\begin{proof}
The proof is similar to \cite[Lemma 4.3]{LS} or \cite[Lemma 3.3]{SW},
so we will not go into details here.
\end{proof}
The following theorem is to provide an approximation of $S_j(t)$.
\begin{theorem}\label{theorem 1 spectral}
For $t\neq 0$, and $x\geq 4$, we have
\begin{equation*}
\begin{split}
S_j(t)=&\frac{1}{\pi}\Im\sum_{n\leq x^3}\frac{C_j(n)\Lambda_x(n)}{n^{\sigma_x+it}\log n}+
O\left((\sigma_x-1/2)\left|\sum_{n\leq x^3} \frac{C_j(n)\Lambda_x(n)}{n^{\sigma_x+it}}\right|\right)\\
&+O\big((\sigma_x-1/2)\log(|t|+t_j+1)\big),
\end{split}
\end{equation*}
where $\sigma_x$ is defined by \eqref{sigma_x spectral}.
\end{theorem}
\begin{proof}
By the definition of $S_j(t)$, we have
\begin{equation*}
\begin{split}
\pi S_j(t)=&-\int_{1/2}^\infty\Im \frac{L'}{L}(\sigma+it,u_j)\mathrm{d}\sigma\\
=&-\int_{\sigma_x}^\infty\Im \frac{L'}{L}(\sigma+it,u_j)\mathrm{d}\sigma
-(\sigma_x-1/2)\Im \frac{L'}{L}(\sigma_x+it,u_j)\\
&+\int_{1/2}^{\sigma_x}\Im\left( \frac{L'}{L}(\sigma_x+it,u_j)-\frac{L'}{L}(\sigma+it,u_j)\right)\mathrm{d}\sigma\\
:=&\, J_1+J_2+J_3.
\end{split}
\end{equation*}
By Lemma \ref{lemma 3 spectral}, we have
\begin{equation*}
\begin{split}
J_1=&-\int_{\sigma_x}^\infty\Im \frac{L'}{L}(\sigma+it,u_j)\mathrm{d}\sigma\\
=&\int_{\sigma_x}^\infty\Im\sum_{n\leq x^3} \frac{C_j(n)\Lambda_x(n)}{n^{\sigma+it}}\mathrm{d}\sigma
+O\left(\left|\sum_{n\leq x^3} \frac{C_j(n)\Lambda_x(n)}{n^{\sigma_x+it}}\right|
\int_{\sigma_x}^\infty x^{1/4-\sigma/2}\mathrm{d}\sigma\right)\\
&+O\left(\log(|t|+t_j+1)\int_{\sigma_x}^\infty x^{1/4-\sigma/2}\mathrm{d}\sigma\right)\\
=&\Im\sum_{n\leq x^3} \frac{C_j(n)\Lambda_x(n)}{n^{\sigma_x+it}\log n}
+O\left(\frac{1}{\log x}\left|\sum_{n\leq x^3} \frac{C_j(n)\Lambda_x(n)}{n^{\sigma_x+it}}\right|\right)
+O\left(\frac{\log(|t|+t_j+1)}{\log x}\right).
\end{split}
\end{equation*}
Taking $\sigma=\sigma_x$ in Lemma \ref{lemma 3 spectral}, we have
\begin{equation*}
\begin{split}
J_2\leq&\,(\sigma_x-1/2)\left|\frac{L'}{L}(\sigma_x+it,u_j)\right|\\
\ll&\,(\sigma_x-1/2)\left|\sum_{n\leq x^3} \frac{C_j(n)\Lambda_x(n)}{n^{\sigma_x+it}}\right|
+(\sigma_x-1/2)\log(|t|+t_j+1).
\end{split}
\end{equation*}
Using Lemma \ref{lemma 2 spectral}, we get
\begin{equation*}
\begin{split}
|J_3|\leq&\sum_\rho\int_{1/2}^{\sigma_x}
 \frac{(\sigma_x-\sigma)|\sigma+\sigma_x-2\beta||t-\gamma|}
{\big((\sigma_x-\beta)^2+(t-\gamma)^2\big)\big((\sigma-\beta)^2+(t-\gamma)^2\big)}\mathrm{d}\sigma
+O(\sigma_x-1/2)\\
\leq&\sum_\rho\frac{\sigma_x-1/2}{(\sigma_x-\beta)^2+(t-\gamma)^2}\int_{1/2}^{\sigma_x}
 \frac{|\sigma+\sigma_x-2\beta||t-\gamma|}
{(\sigma-\beta)^2+(t-\gamma)^2}\mathrm{d}\sigma
+O(\sigma_x-1/2).
\end{split}
\end{equation*}
If $|\beta-1/2|\leq\frac{1}{2}(\sigma_x-1/2)$, then for $1/2\leq\sigma\leq\sigma_x$,
\begin{equation*}
\begin{split}
|\sigma+\sigma_x-2\beta|&=|(\sigma-1/2)+(\sigma_x-1/2)-2(\beta-1/2)|\\
&\leq |\sigma-1/2|+|\sigma_x-1/2|+2|\beta-1/2|\leq 3(\sigma_x-1/2).
\end{split}
\end{equation*}
Thus,
\begin{equation*}
\begin{split}
\int_{1/2}^{\sigma_x}
 \frac{|\sigma+\sigma_x-2\beta||t-\gamma|}
{(\sigma-\beta)^2+(t-\gamma)^2}\mathrm{d}\sigma
&\leq 3(\sigma_x-1/2)\int_{-\infty}^\infty\frac{|t-\gamma|}
{(\sigma-\beta)^2+(t-\gamma)^2}\mathrm{d}\sigma\\
&\leq 10(\sigma_x-1/2).
\end{split}
\end{equation*}
If $|\beta-1/2|>\frac{1}{2}(\sigma_x-1/2)$, then
\bna
|t-\gamma|>\frac{x^{3|\beta-1/2|}}{\log x}>3|\beta-1/2|
\ena
and for $1/2\leq\sigma\leq\sigma_x$,
\begin{equation*}
\begin{split}
|\sigma+\sigma_x-2\beta|\leq |\sigma-1/2|+|\sigma_x-1/2|+2|\beta-1/2|\leq 6|\beta-1/2|.
\end{split}
\end{equation*}
Thus,
\begin{equation*}
\begin{split}
\int_{1/2}^{\sigma_x}
 \frac{|\sigma+\sigma_x-2\beta||t-\gamma|}
{(\sigma-\beta)^2+(t-\gamma)^2}\mathrm{d}\sigma
&<\int_{1/2}^{\sigma_x}\frac{|\sigma+\sigma_x-2\beta|}
{|t-\gamma|}\mathrm{d}\sigma\\
&\leq \int_{1/2}^{\sigma_x}\frac{6|\beta-1/2|}
{3|\beta-1/2|}\mathrm{d}\sigma=2(\sigma_x-1/2).
\end{split}
\end{equation*}
By Lemma \ref{lemma 3 spectral}, we obtain
\begin{equation*}
\begin{split}
|J_3|\leq&10(\sigma_x-1/2)\sum_\rho\frac{\sigma_x-1/2}{(\sigma_x-\beta)^2+(t-\gamma)^2}
+O(\sigma_x-1/2)\\
=&O\left((\sigma_x-1/2)\left|\sum_{n\leq x^3} \frac{C_j(n)\Lambda_x(n)}{n^{\sigma_x+it}}\right|\right)
+O\big((\sigma_x-1/2)\log(|t|+t_j+1)\big).
\end{split}
\end{equation*}
Thus we complete the proof of lemma.
\end{proof}
As a corollary of Theorem \ref{theorem 1 spectral}, we give an analog which was established by
Selberg \cite{Selberg1} and Hejhal--Luo \cite{HL}.
\begin{theorem}
Under the \textrm{GRH}, we have
\bna
S_j(t)\ll\frac{\log(|t|+t_j+1)}{\log\log(|t|+t_j+1)}.
\ena
\end{theorem}
\begin{proof}
We have $\sigma_x=\frac{1}{2}+\frac{10}{\log x}$ by mean of the GRH.
By \eqref{C_j prime power} and \eqref{lambda_j(n) bound},
\bea\label{bound 3}
|C_j(p^m)\Lambda_x(p^m)|\ll p^{m\vartheta}\log p.
\eea
Thus,
\bna
\left|\Im\sum_{n\leq x^3}\frac{C_j(n)\Lambda_x(n)}{n^{\sigma_x+it}\log n}\right|
\ll \sum_{p\leq x^3}p^{\vartheta-1/2}\ll \frac{x^{3(\vartheta+1/2)}}{\log x},
\ena
and
\bna
(\sigma_x-1/2)\left|\sum_{n\leq x^3} \frac{C_j(n)\Lambda_x(n)}{n^{\sigma_x+it}}\right|
\ll \frac{1}{\log x}\sum_{p\leq x^3}p^{\vartheta-1/2}\log p\ll \frac{x^{3(\vartheta+1/2)}}{\log x}.
\ena
Then the theorem follows by taking $x=\big(\log(|t|+t_j+1)\big)^{\frac{64}{117}}$ in Theorem \ref{theorem 1 spectral}.
\end{proof}
\section{Proof of Theorem \ref{main-theorem}}
\begin{lemma}\label{lemma 6}
For fixed $t>0$ and large parameter $M$ with $T^\varepsilon\leq M\leq T^{1-\varepsilon}$.
For $T$ sufficiently large and $x=T^{\delta/3}$ with $0<\delta<\frac{3\theta}{8n\ell+3}$, $\ell\in\{0,1,2\}$, we have
\begin{equation*}
\begin{split}
\sideset{}{^\prime}\sum_{j=1}^\infty e^{-\frac{(t_j-T)^2}{M^2}} |\nu_j(1)|^2(\sigma_{x,j}-1/2)^{4n}x^{4 n\ell(\sigma_{x,j}-1/2)}
\ll_{t,n,\delta,\ell}\frac{MT}{(\log T)^{4n}},
\end{split}
\end{equation*}
where $\sigma_{x,j}$ is defined by \eqref{sigma_x spectral} and $\theta$ is as in Lemma \ref{zero-density spectral}.
\end{lemma}
\begin{proof}
For $\ell\in\{0,1,2\}$, by the definition of $\sigma_{x,j}$, we have
\begin{equation*}
\begin{split}
&(\sigma_{x,j}-1/2)^{4n}x^{4n\ell(\sigma_{x,j}-1/2)}\\
\leq&\left(\frac{10}{\log x}\right)^{4n}x^{40n\ell/\log x}
+2^{4n+1}\sum_{\beta>\frac{1}{2}+\frac{5}{\log x}\atop |t-\gamma|\leq\frac{x^{3|\beta-1/2|}}{\log x}}
(\beta-1/2)^{4n}x^{8n\ell(\beta-1/2)}.
\end{split}
\end{equation*}
On the other hand,
\begin{equation*}
\begin{split}
&\sum_{\beta>\frac{1}{2}+\frac{5}{\log x}\atop |t-\gamma|\leq\frac{x^{3|\beta-1/2|}}{\log x}}
(\beta-1/2)^{4n}x^{8n\ell(\beta-1/2)}\\
\leq&\sum_{k=5}^{\frac{1}{2}\lfloor\log x\rfloor}\left(\frac{k+1}{\log x}\right)^{4n}
x^{8n\ell\frac{k+1}{\log x}}\sum_{\frac{1}{2}+\frac{k}{\log x}<\beta\leq\frac{1}{2}+\frac{k+1}{\log x}
\atop |t-\gamma|\leq\frac{x^{3|\beta-1/2|}}{\log x}}1\\
\leq&\frac{1}{(\log x)^{4n}}\sum_{k=5}^{\frac{1}{2}\lfloor\log x\rfloor}\left(k+1\right)^{4n}
e^{8n\ell(k+1)}N\left(\frac{1}{2}+\frac{k}{\log x}, |t|+\frac{e^{3(k+1)}}{\log x};u_j\right).
\end{split}
\end{equation*}
By Lemma \ref{zero-density spectral} and the bound $\mathcal{H}\ll MT$, we have
\begin{equation*}
\begin{split}
&\sideset{}{^\prime}\sum_{j=1}^\infty e^{-\frac{(t_j-T)^2}{M^2}} |\nu_j(1)|^2
\sum_{\beta>\frac{1}{2}+\frac{5}{\log x}\atop |t-\gamma|\leq\frac{x^{3|\beta-1/2|}}{\log x}}
(\beta-1/2)^{4n}x^{8n\ell(\beta-1/2)}\\
\leq&\frac{1}{(\log x)^{4n}}\sum_{k=5}^{\frac{1}{2}\lfloor\log x\rfloor}\left(k+1\right)^{4n}
e^{8n\ell(k+1)}\sideset{}{^\prime}\sum_{j=1}^\infty e^{-\frac{(t_j-T)^2}{M^2}} |\nu_j(1)|^2
N\left(\frac{1}{2}+\frac{k}{\log x}, |t|+\frac{e^{3(k+1)}}{\log x};u_j\right)\\
\ll&\frac{1}{(\log x)^{4n}}\sum_{k=5}^{\infty}\left(k+1\right)^{4n}
e^{8n\ell(k+1)}\mathcal{H}\left(|t|+\frac{e^{3(k+1)}}{\log x}\right) T^{-\theta\frac{k}{\log x}}\log T\\
\ll&_{t,n,\delta,\ell}\,\frac{MT\log T}{(\log x)^{4n+1}}\sum_{k=5}^{\infty}\left(k+1\right)^{4n}
e^{(8n\ell+3-\frac{3\theta}{\delta})k}\\
\ll&_{t,n,\delta,\ell}\,\frac{MT}{(\log T)^{4n}}
\end{split}
\end{equation*}
provided that $0<\delta<\frac{3\theta}{8n\ell+3}$.
In addition, by Lemma \ref{corollary1},
\begin{equation*}
\begin{split}
\sideset{}{^\prime}\sum_{j=1}^\infty e^{-\frac{(t_j-T)^2}{M^2}} |\nu_j(1)|^2
\left(\frac{10}{\log x}\right)^{4n}x^{40n\ell/\log x}
\ll_{n,\ell}\frac{TM}{(\log x)^{4n}}\ll_{n,\ell,\delta}\frac{TM}{(\log T)^{4n}}.
\end{split}
\end{equation*}
Thus we complete the proof of lemma.
\end{proof}

For a positive parameter $x$ (to be chosen later), let
\bna
M_j(t):=\frac{1}{\pi}\Im\sum_{p\leq x^3}\frac{C_j(p)}{p^{1/2+it}},\quad
\text{and}\quad
R_j(t):=S_j(t)-M_j(t),
\ena
recall here $C_j(p)=\alpha_j(p)+\beta_j(p)=\lambda_j(p)$.
\begin{lemma}\label{lemma 5 spectral} We have
\begin{equation*}
\begin{split}
R_j(t)=&O\left(\left|\Im\sum_{p\leq x^3}
\frac{\lambda_j(p)(\Lambda_x(p)-\Lambda(p))}{p^{1/2+it}\log p}\right|\right)
+O\left(\left|\Im\sum_{p\leq x^{3/2}}\frac{\lambda_j(p^2)\Lambda_x(p^2)}{p^{1+2it}\log p}\right|\right)\\
&+O\left((\sigma_x-1/2)x^{\sigma_x-1/2}\int_{1/2}^\infty
x^{1/2-\sigma}\left|\sum_{p\leq x^3}
\frac{\lambda_j(p)\Lambda_x(p)\log(xp)}{p^{\sigma+it}}\right|\mathrm{d}\sigma\right)\\
&+O\left((\sigma_x-1/2)x^{2\sigma_x-1}\int_{1/2}^\infty x^{1-2\sigma}\left|\sum_{p\leq x^{3/2}}
\frac{\lambda_j(p^2)\Lambda_x(p^2)\log(xp)}{p^{2\sigma+2it}}\right|\mathrm{d}\sigma\right)\\
&+O\big((\sigma_x-1/2)\log(|t|+t_j+1)\big)+O\big((\sigma_x-1/2)\log x\big).
\end{split}
\end{equation*}
\end{lemma}
\begin{proof}
By Theorem \ref{theorem 1 spectral},
\begin{equation*}
\begin{split}
R_j(t)=&\frac{1}{\pi}\Im\sum_{p\leq x^3}
\frac{C_j(p)(\Lambda_x(p)p^{1/2-\sigma_x}-\Lambda(p))}{p^{1/2+it}\log p}
+\frac{1}{\pi}\Im\sum_{m=2}^\infty\sum_{p^m\leq x^3}\frac{C_j(p^m)\Lambda_x(p^m)}{p^{m(\sigma_x+it)}\log p^m}\\
&+O\left((\sigma_x-1/2)\left|\sum_{m=1}^\infty\sum_{p^m\leq x^3} \frac{C_j(p^m)
\Lambda_x(p^m)}{p^{m(\sigma_x+it)}}\right|\right)
+O\big((\sigma_x-1/2)\log(|t|+t_j+1)\big).
\end{split}
\end{equation*}
Using the bound \eqref{bound 3}, we deduce that
\begin{equation*}
\begin{split}
\sum_{m=3}^\infty\sum_{p^m\leq x^3}\frac{C_j(p^m)\Lambda_x(p^m)}{p^{m(\sigma_x+it)}}=O(1),\quad\text{and}\quad
\sum_{m=3}^\infty\sum_{p^m\leq x^3}\frac{C_j(p^m)\Lambda_x(p^m)}{p^{m(\sigma_x+it)}\log p}=O(1).
\end{split}
\end{equation*}
Thus,
\begin{equation*}
\begin{split}
R_j(t)=&\frac{1}{\pi}\Im\sum_{p\leq x^3}
\frac{\lambda_j(p)(\Lambda_x(p)-\Lambda(p))}{p^{1/2+it}\log p}
-\frac{1}{\pi}\Im\sum_{p\leq x^3}
\frac{\lambda_j(p)\Lambda_x(p)(1-p^{1/2-\sigma_x})}{p^{1/2+it}\log p}\\
&+\frac{1}{\pi}\Im\sum_{p\leq x^{3/2}}\frac{C_j(p^2)\Lambda_x(p^2)}{p^{1+2it}\log p^2}
+\frac{1}{\pi}\Im\sum_{p\leq x^{3/2}}\frac{C_j(p^2)\Lambda_x(p^2)(p^{1-2\sigma_x}-1)}{p^{1+2it}\log p^2}\\
&+O\left((\sigma_x-1/2)\left|\sum_{p\leq x^3} \frac{\lambda_j(p)
\Lambda_x(p)}{p^{\sigma_x+it}}\right|\right)
+O\left((\sigma_x-1/2)\left|\sum_{p\leq x^{3/2}}\frac{C_j(p^2)\Lambda_x(p^2)}{p^{2(\sigma_x+it)}}\right|\right)\\
&+O\big((\sigma_x-1/2)\log(|t|+t_j+1)\big)+O(1)\\
:=&\sum_{i=1}^8S_i.
\end{split}
\end{equation*}
Recall that $C_j(p^2)=\lambda_j(p^2)-1$ and the definition of $\Lambda_x$, we have
\begin{equation*}
\begin{split}
S_3=&\frac{1}{\pi}\Im\sum_{p\leq x^{3/2}}\frac{\lambda_j(p^2)\Lambda_x(p^2)}{p^{1+2it}\log p^2}
-\frac{1}{\pi}\Im\sum_{p\leq x^{3/2}}\frac{\Lambda_x(p^2)}{p^{1+2it}\log p^2}\\
=&O\left(\left|\Im\sum_{p\leq x^{3/2}}\frac{\lambda_j(p^2)\Lambda_x(p^2)}{p^{1+2it}\log p}\right|\right)
+O\left(\left|\sum_{p\leq x^{1/2}}\frac{\Lambda(p^2)}{p^{1+2it}\log p}\right|\right)\\
&+O\left(\left|\sum_{x^{1/2}\leq p\leq x}\frac{\Lambda_x(p^2)}{p^{1+2it}\log p}\right|\right)
+O\left(\left|\sum_{x\leq p\leq x^{3/2}}\frac{\Lambda_x(p^2)}{p^{1+2it}\log p}\right|\right)\\
=&O\left(\left|\Im\sum_{p\leq x^{3/2}}\frac{\lambda_j(p^2)\Lambda_x(p^2)}{p^{1+2it}\log p}\right|\right)+O(1),
\end{split}
\end{equation*}
Similarly, we have
\begin{equation*}
\begin{split}
S_4
=&O\left(\left|\Im\sum_{p\leq x^{3/2}}\frac{\lambda_j(p^2)\Lambda_x(p^2)(p^{1-2\sigma_x}-1)}{p^{1+2it}\log p}\right|\right)+O(1),
\end{split}
\end{equation*}
and
\begin{equation*}
\begin{split}
S_6
=&O\left((\sigma_x-1/2)\left|\sum_{p\leq x^{3/2}}\frac{\lambda_j(p^2)\Lambda_x(p^2)}{p^{2(\sigma_x+it)}}\right|\right)+
O\big((\sigma_x-1/2)\log x\big).
\end{split}
\end{equation*}
Note that for $1/2\leq a\leq \sigma_x$,
\begin{equation*}
\begin{split}
\left|\sum_{p\leq x^3} \frac{\lambda_j(p)\Lambda_x(p)}{p^{a+it}}\right|
&=x^{a-1/2}\left|\int_a^\infty x^{1/2-\sigma}\sum_{p\leq x^3}
\frac{\lambda_j(p)\Lambda_x(p)\log(xp)}{p^{\sigma+it}}\mathrm{d}\sigma\right|\\
&\leq x^{\sigma_x-1/2}\int_{1/2}^\infty x^{1/2-\sigma}\left|\sum_{p\leq x^3}
\frac{\lambda_j(p)\Lambda_x(p)\log(xp)}{p^{\sigma+it}}\right|\mathrm{d}\sigma.
\end{split}
\end{equation*}
Thus,
\begin{equation*}
\begin{split}
S_2&\ll \left|\sum_{p\leq x^3}
\frac{\lambda_j(p)\Lambda_x(p)(1-p^{1/2-\sigma_x})}{p^{1/2+it}\log p}\right|
=\left|\int_{1/2}^{\sigma_x} \sum_{p\leq x^3} \frac{\lambda_j(p)\Lambda_x(p)}{p^{a+it}}\mathrm{d}a\right|\\
&\leq (\sigma_x-1/2)x^{\sigma_x-1/2}\int_{1/2}^\infty x^{1/2-\sigma}\left|\sum_{p\leq x^3}
\frac{\lambda_j(p)\Lambda_x(p)\log(xp)}{p^{\sigma+it}}\right|\mathrm{d}\sigma,
\end{split}
\end{equation*}
and
\begin{equation*}
\begin{split}
S_5\ll(\sigma_x-1/2)\left|\sum_{p\leq x^3} \frac{\lambda_j(p)
\Lambda_x(p)}{p^{\sigma_x+it}}\right|\leq (\sigma_x-1/2)x^{\sigma_x-1/2}\int_{1/2}^\infty
x^{1/2-\sigma}\left|\sum_{p\leq x^3}
\frac{\lambda_j(p)\Lambda_x(p)\log(xp)}{p^{\sigma+it}}\right|\mathrm{d}\sigma.
\end{split}
\end{equation*}
Similarly, for $1/2\leq a\leq \sigma_x$,
\begin{equation*}
\begin{split}
\left|\sum_{p\leq x^{3/2}} \frac{\lambda_j(p^2)\Lambda_x(p^2)}{p^{2a+2it}}\right|
&=2x^{2a-1}\left|\int_a^\infty x^{1-2\sigma}\sum_{p\leq x^{3/2}}
\frac{\lambda_j(p^2)\Lambda_x(p^2)\log(xp)}{p^{2\sigma+2it}}\mathrm{d}\sigma\right|\\
&\leq 2x^{2\sigma_x-1}\int_{1/2}^\infty x^{1-2\sigma}\left|\sum_{p\leq x^{3/2}}
\frac{\lambda_j(p^2)\Lambda_x(p^2)\log(xp)}{p^{2\sigma+2it}}\right|\mathrm{d}\sigma.
\end{split}
\end{equation*}
Thus, we have
\begin{equation*}
\begin{split}
&\left|\sum_{p\leq x^{3/2}}\frac{\lambda_j(p^2)\Lambda_x(p^2)(p^{1-2\sigma_x}-1)}{p^{1+2it}\log p}\right|
=2\left|\int_{1/2}^{\sigma_x}\sum_{p\leq x^{3/2}}\frac{\lambda_j(p^2)\Lambda_x(p^2)}{p^{2a+2it}}\mathrm{d}a\right|\\
&\leq 4(\sigma_x-1/2)x^{2\sigma_x-1}\int_{1/2}^\infty x^{1-2\sigma}\left|\sum_{p\leq x^{3/2}}
\frac{\lambda_j(p^2)\Lambda_x(p^2)\log(xp)}{p^{2\sigma+2it}}\right|\mathrm{d}\sigma,
\end{split}
\end{equation*}
and
\begin{equation*}
\begin{split}
(\sigma_x-1/2)\left|\sum_{p\leq x^{3/2}}\frac{\lambda_j(p^2)\Lambda_x(p^2)}{p^{2(\sigma_x+it)}}\right|
\leq 2(\sigma_x-1/2)x^{2\sigma_x-1}\int_{1/2}^\infty x^{1-2\sigma}\left|\sum_{p\leq x^{3/2}}
\frac{\lambda_j(p^2)\Lambda_x(p^2)\log(xp)}{p^{2\sigma+2it}}\right|\mathrm{d}\sigma.
\end{split}
\end{equation*}
Thus we complete the proof of lemma.
\end{proof}
\begin{lemma}\label{lemma 6 spectral}
For fixed $t>0$ and large parameter $M$ with $T^\varepsilon\leq M\leq T^{1-\varepsilon}$.
For $T$ sufficiently large and $x=T^{\delta/3}$ with sufficiently small $\delta>0$, we have
\begin{equation}\label{M_j(t) moment spectral}
\begin{split}
\sideset{}{^\prime}\sum_{j=1}^\infty e^{-\frac{(t_j-T)^2}{M^2}} |\nu_j(1)|^2M_j(t)^n
=\frac{MT}{4\pi^{3/2}}C_n(\log\log T)^{\frac{n}{2}}
+O_n\big(MT(\log\log T)^{\frac{n}{2}-1}\big),
\end{split}
\end{equation}
and
\begin{equation}\label{R_j(t) moment spectral}
\begin{split}
\sideset{}{^\prime}\sum_{j=1}^\infty e^{-\frac{(t_j-T)^2}{M^2}}|\nu_j(1)|^2 |R_j(t)|^{2n}=O(MT).
\end{split}
\end{equation}
Here $C_n$ is defined by
\bna
C_n=\begin{cases}
\frac{n!}{(n/2)!(2\pi)^{n}}, \,&\textit{if }\, n\,\, \text{is even},\\
0, \,& \textit{if }\, n \,\,\text{is odd}.
\end{cases}
\ena
\end{lemma}
\begin{proof}
Recall that
\bna
M_j(t)=\frac{1}{\pi}\Im\sum_{p\leq x^3}\frac{\lambda_j(p)}{p^{1/2+it}}=
\frac{-i}{2\pi}\left(\sum_{p\leq x^3}\frac{\lambda_j(p)}{p^{1/2+it}}
-\sum_{p\leq x^3}\frac{\lambda_j(p)}{p^{1/2-it}}\right).
\ena
Set $x=T^{\delta/3}$ for a suitably small $\delta>0$. Thus we have that the general term of
\bna
M_j(t)^n=\frac{(-i)^n}{(2\pi)^n}\left(\sum_{p\leq T^{\delta}}\frac{\lambda_j(p)}{p^{1/2+it}}
-\sum_{p\leq T^{\delta}}\frac{\lambda_j(p)}{p^{1/2-it}}\right)^n
\ena
has the form
\bea\label{general term spectral}
\frac{\lambda_j(p_1)^{m(p_1)+n(p_1)}}{p_1^{m(p_1)(1/2+it)}(-1)^{n(p_1)}p_1^{n(p_1)(1/2-it)}}
\times\cdots\times
\frac{\lambda_j(p_r)^{m(p_r)+n(p_r)}}{p_r^{m(p_r)(1/2+it)}(-1)^{n(p_r)}p_r^{n(p_r)(1/2-it)}},
\eea
where $p_1<p_2<\cdots<p_r\leq T^\delta$, $m(p_j)+n(p_j)\geq 1$, and $\sum_{j=1}^r(m(p_j)+n(p_j))=n$.

Now we discuss the contribution from the general term \eqref{general term spectral} in the following cases.

Case\,(\uppercase\expandafter{\romannumeral 1}):
\emph{In the general term \eqref{general term spectral}, $m(p_{j_0})\not\equiv n(p_{j_0})(\bmod\,2)$ for some $j_0$.}
With the Hecke relation \eqref{the Hecke relation 1}, the numerator of this general term \eqref{general term spectral}
can be written as a sum whose terms are all of the form $\lambda_j(m)\lambda_j(n)$ where $m\neq n$.
Since for these terms $\delta_{m,n}=0$, we see from Lemma \ref{corollary1} that the contribution from the general term
\eqref{general term spectral} to the moment \eqref{M_j(t) moment spectral} is
\begin{equation*}
\begin{split}
&\ll_n\sum_{p_1<p_2<\cdots<p_r\leq T^\delta}p_1^{-\frac{1}{2}(m(p_1)+n(p_1))}\cdots
p_r^{-\frac{1}{2}(m(p_r)+n(p_r))}MT^{\varepsilon}\\
&\ll_n\bigg(\sum_{p\leq T^\delta}p^{-\frac{1}{2}}\bigg)^r MT^{\varepsilon}
\ll_nM T^{\frac{\delta n}{2}(1+\varepsilon)}\ll_n MT^{\frac{1}{2}+\varepsilon},
\end{split}
\end{equation*}
upon taking $0<\delta<\frac{1}{n}$.

Case\,(\uppercase\expandafter{\romannumeral 2}):
\emph{In the general term \eqref{general term spectral}, $m(p_{j})\equiv n(p_{j})(\bmod\,2)$ for all $j$,
and $m(p_{j_0})+n(p_{j_0})\geq4$ for some $j_0$.}
In this case, $n$ is even and $n>2r$, so $r<\frac{n}{2}$.
By using Hecke relation \eqref{the Hecke relation 1},
one sees that it is possible to have resulting terms $\lambda_j(m)\lambda_j(n)$ with $m=n$.
Using Lemma \ref{corollary1}, the contribution to the moment \eqref{M_j(t) moment spectral} such terms is bounded by
the diagonal term of size $\asymp MT$. The contribution from the off-diagonal terms is clearly negligible compared
to $MT$, given that $0<\delta<\frac{1}{n}$.
Using the classical result
\bea\label{classic bound 2}
\sum_{p\leq x}\frac{1}{p}=\log\log x+O(1),
\eea
we have the contribution from this general term to the moment \eqref{M_j(t) moment spectral} is at most
\begin{equation*}
\begin{split}
&\ll_nMT\sum_{p_1<p_2<\cdots<p_r\leq T^\delta}p_1^{-\frac{1}{2}(m(p_1)+n(p_1))}\cdots
p_r^{-\frac{1}{2}(m(p_r)+n(p_r))}\\
&\ll_nMT\bigg(\sum_{p\leq T^\delta}p^{-1}\bigg)^r
\ll_nMT (\log\log T)^{\frac{n}{2}-1}.
\end{split}
\end{equation*}

It remains to discuss the following: In the general term \eqref{general term spectral}, $m(p_{j})\equiv n(p_{j})(\bmod\,2)$
and $m(p_j)+n(p_j)=2$ for all $j$. Clearly $n$ must be even, say $n=2m$ and so $r=m$.

Case\,(\uppercase\expandafter{\romannumeral 3}):
\emph{$m(p_{j_0})=2$ or $n(p_{j_0})=2$ for some $j_0$.}
We deduce that the contribution from these terms to \eqref{M_j(t) moment spectral} is
at most $O\big(MT(\log\log T)^{\frac{n}{2}-1}\big)$ by using \eqref{classic bound 2} and the convergence of $\sum_{p}p^{-1-it}$ for $t\neq 0$.

Case\,(\uppercase\expandafter{\romannumeral 4}):
\emph{$m(p_{j})=n(p_{j})=1$ for all $j$.}
The contribution from these terms to \eqref{M_j(t) moment spectral} is
\begin{equation*}
\begin{split}
\sideset{}{^\prime}\sum_{j=1}^\infty e^{-\frac{(t_j-T)^2}{M^2}} |\nu_j(1)|^2\binom{2m}{m}(m!)(m!)\frac{(-1)^m(-i)^{2m}}{(2\pi)^{2m}}
\sum_{p_1<p_2<\cdots<p_m\leq T^\delta}\frac{\lambda_j(p_1)^2\cdots\lambda_j(p_m)^2}{p_1\cdots p_m}.
\end{split}
\end{equation*}
By the Hecke relations, we have $\lambda_j(p_1)^2\cdots\lambda_j(p_m)^2=\lambda_j(p_1\cdots p_m)\lambda_j(p_1\cdots p_m)$.
Applying Lemma \ref{corollary1} and \eqref{classic bound 2}, we have the above equals
\begin{equation*}
\begin{split}
&\frac{(2m)!}{(2\pi)^{2m}}\sum_{p_1<p_2<\cdots<p_m\leq T^\delta}\frac{1}{p_1\cdots p_m}
\cdot\left(\frac{MT}{4\pi^{3/2}}+O_\varepsilon\left(MT^\varepsilon\right)\right)\\
=&\frac{(2m)!}{m!(2\pi)^{2m}}\sum_{p_1,p_2,\ldots,p_m\leq T^\delta\atop p_j \,\text{distinct}}\frac{1}{p_1\cdots p_m}
\cdot\left(\frac{MT}{4\pi^{3/2}}+O_\varepsilon\left(MT^\varepsilon\right)\right)\\
=&\frac{MT}{4\pi^{3/2}}\frac{(2m)!}{m!(2\pi)^{2m}}\sum_{p_1,p_2,\ldots,p_m\leq T^\delta\atop p_j \,\text{distinct}}\frac{1}{p_1\cdots p_m}
+O_\varepsilon\bigg(MT^\varepsilon\sum_{p_1,p_2,\ldots,p_m\leq T^\delta\atop p_j\,\text{distinct}}\frac{1}{p_1\cdots p_m}\bigg)\\
=&\frac{MT}{4\pi^{3/2}}\frac{(2m)!}{m!(2\pi)^{2m}}(\log\log T+O(1))^m
+O_\varepsilon\big(MT^\varepsilon(\log\log T)^m\big)\\
=&\frac{MT}{4\pi^{3/2}}\frac{(2m)!}{m!(2\pi)^{2m}}(\log\log T)^m
+O_m\big(MT(\log\log T)^{m-1}\big).
\end{split}
\end{equation*}
Now the asymptotic formula \eqref{M_j(t) moment spectral} follows from
Cases (\uppercase\expandafter{\romannumeral 1})--(\uppercase\expandafter{\romannumeral 4}).

Next we are ready to prove \eqref{R_j(t) moment spectral}.
Recall that $\sigma_x=\sigma_{x,j}$ depending on $u_j$.
By Lemma \ref{lemma 5 spectral},
\begin{equation}\label{R}
\begin{split}
&\sideset{}{^\prime}\sum_{j=1}^\infty e^{-\frac{(t_j-T)^2}{M^2}}|\nu_j(1)|^2 |R_j(t)|^{2n}\\
\ll&\sideset{}{^\prime}\sum_{j=1}^\infty e^{-\frac{(t_j-T)^2}{M^2}}|\nu_j(1)|^2\left|\Im\sum_{p\leq x^3}
\frac{\lambda_j(p)(\Lambda_x(p)-\Lambda(p))}{p^{1/2+it}\log p}\right|^{2n}
+\sideset{}{^\prime}\sum_{j=1}^\infty e^{-\frac{(t_j-T)^2}{M^2}}|\nu_j(1)|^2
\left|\Im\sum_{p\leq x^{3/2}}\frac{\lambda_j(p^2)\Lambda_x(p^2)}{p^{1+2it}\log p}\right|^{2n}\\
&+\sideset{}{^\prime}\sum_{j=1}^\infty e^{-\frac{(t_j-T)^2}{M^2}}|\nu_j(1)|^2
(\sigma_x-1/2)^{2n}x^{2n(\sigma_x-1/2)}\left(\int_{1/2}^\infty
x^{1/2-\sigma}\left|\sum_{p\leq x^3}
\frac{\lambda_j(p)\Lambda_x(p)\log(xp)}{p^{\sigma+it}}\right|\mathrm{d}\sigma\right)^{2n}\\
&+\sideset{}{^\prime}\sum_{j=1}^\infty e^{-\frac{(t_j-T)^2}{M^2}}|\nu_j(1)|^2
(\sigma_x-1/2)^{2n}x^{2n(2\sigma_x-1)}\left(\int_{1/2}^\infty
x^{1-2\sigma}\left|\sum_{p\leq x^{3/2}}
\frac{\lambda_j(p^2)\Lambda_x(p^2)\log(xp)}{p^{2\sigma+2it}}\right|\mathrm{d}\sigma\right)^{2n}\\
&+\sideset{}{^\prime}\sum_{j=1}^\infty e^{-\frac{(t_j-T)^2}{M^2}}|\nu_j(1)|^2(\sigma_x-1/2)^{2n}\big(\log(|t|+t_j+1)\big)^{2n}\\
&+\sideset{}{^\prime}\sum_{j=1}^\infty e^{-\frac{(t_j-T)^2}{M^2}}|\nu_j(1)|^2(\sigma_x-1/2)^{2n}\big(\log x\big)^{2n}.
\end{split}
\end{equation}
For the first two terms in \eqref{R}, we have
\begin{equation*}
\begin{split}
&\left|\Im\sum_{p\leq x^3}
\frac{\lambda_j(p)(\Lambda_x(p)-\Lambda(p))}{p^{1/2+it}\log p}\right|^{2n}
\ll_n\left(\sum_{p\leq x^3}
\frac{\lambda_j(p)(\Lambda_x(p)-\Lambda(p))}{p^{1/2+it}\log p}-\sum_{p\leq x^3}
\frac{\lambda_j(p)(\Lambda_x(p)-\Lambda(p))}{p^{1/2-it}\log p}\right)^{2n},
\end{split}
\end{equation*}
and
\begin{equation}\label{aim3}
\begin{split}
&\left|\Im\sum_{p\leq x^{3/2}}\frac{\lambda_j(p^2)\Lambda_x(p^2)}{p^{1+2it}\log p}\right|^{2n}
\ll_n\left(\sum_{p\leq x^{3/2}}\frac{\lambda_j(p^2)\Lambda_x(p^2)}{p^{1+2it}\log p}-
\sum_{p\leq x^{3/2}}\frac{\lambda_j(p^2)\Lambda_x(p^2)}{p^{1-2it}\log p}\right)^{2n}.
\end{split}
\end{equation}
Here we use the fact $|\Lambda_x(p)-\Lambda(p)|=O\left(\frac{\log^3p}{\log^2x}\right)$ and obtain
these two terms is of $O(TM)$ by the same and easier argument than \eqref{M_j(t) moment spectral}.
The last two terms in \eqref{R} is of $O(TM)$ by Cauchy--Schwarz inequality and Lemma \ref{lemma 6}.
For the third term in \eqref{R},
it follows from Cauchy--Schwarz inequality that
\begin{equation}\label{after Cauchy}
\begin{split}
&\sideset{}{^\prime}\sum_{j=1}^\infty e^{-\frac{(t_j-T)^2}{M^2}}|\nu_j(1)|^2
(\sigma_x-1/2)^{2n}x^{2n(\sigma_x-1/2)}\left(\int_{1/2}^\infty
x^{1/2-\sigma}\left|\sum_{p\leq x^3}
\frac{\lambda_j(p)\Lambda_x(p)\log(xp)}{p^{\sigma+it}}\right|\mathrm{d}\sigma\right)^{2n}\\
&\leq \left(\sideset{}{^\prime}\sum_{j=1}^\infty e^{-\frac{(t_j-T)^2}{M^2}}|\nu_j(1)|^2(\sigma_x-1/2)^{4n}x^{4n(\sigma_x-1/2)}\right)^{1/2}\\
&\times\left(\sideset{}{^\prime}\sum_{j=1}^\infty e^{-\frac{(t_j-T)^2}{M^2}}|\nu_j(1)|^2\left(\int_{1/2}^\infty
x^{1/2-\sigma}\left|\sum_{p\leq x^3}
\frac{\lambda_j(p)\Lambda_x(p)\log(xp)}{p^{\sigma+it}}\right|\mathrm{d}\sigma\right)^{4n}\right)^{1/2}.
\end{split}
\end{equation}
By H\"{o}lder's inequality with the exponents $4n/(4n-1)$ and $4n$,
\begin{equation*}
\begin{split}
&\left(\int_{1/2}^\infty
x^{1/2-\sigma}\left|\sum_{p\leq x^3}
\frac{\lambda_j(p)\Lambda_x(p)\log(xp)}{p^{\sigma+it}}\right|\mathrm{d}\sigma\right)^{4n}\\
\leq&\left(\int_{1/2}^\infty
x^{1/2-\sigma}\mathrm{d}\sigma\right)^{4n-1}\left(\int_{1/2}^\infty
x^{1/2-\sigma}\left|\sum_{p\leq x^3}
\frac{\lambda_j(p)\Lambda_x(p)\log(xp)}{p^{\sigma+it}}\right|^{4n}\mathrm{d}\sigma\right)\\
=&\frac{1}{(\log x)^{4n-1}}\int_{1/2}^\infty
x^{1/2-\sigma}\left|\sum_{p\leq x^3}
\frac{\lambda_j(p)\Lambda_x(p)\log(xp)}{p^{\sigma+it}}\right|^{4n}\mathrm{d}\sigma.
\end{split}
\end{equation*}
Note that
\begin{equation*}
\begin{split}
&\left|\sum_{p\leq x^3}
\frac{\lambda_j(p)\Lambda_x(p)\log(xp)}{p^{\sigma+it}}\right|^{4n}\\
&\ll_n \left|\Re\sum_{p\leq x^3}
\frac{\lambda_j(p)\Lambda_x(p)\log(xp)}{p^{\sigma+it}}\right|^{4n}
+\left|\Im\sum_{p\leq x^3}
\frac{\lambda_j(p)\Lambda_x(p)\log(xp)}{p^{\sigma+it}}\right|^{4n}\\
&\ll_n \left(\sum_{p\leq x^3}
\frac{\lambda_j(p)\Lambda_x(p)\log(xp)(p^{1/2-\sigma})}{p^{1/2+it}}+\sum_{p\leq x^3}
\frac{\lambda_j(p)\Lambda_x(p)\log(xp)(p^{1/2-\sigma})}{p^{1/2-it}}\right)^{4n}\\
&\quad+\left(\sum_{p\leq x^3}
\frac{\lambda_j(p)\Lambda_x(p)\log(xp)(p^{1/2-\sigma})}{p^{1/2+it}}-\sum_{p\leq x^3}
\frac{\lambda_j(p)\Lambda_x(p)\log(xp)(p^{1/2-\sigma})}{p^{1/2-it}}\right)^{4n}.
\end{split}
\end{equation*}
By using the same argument as in proving \eqref{M_j(t) moment spectral}, we get
\begin{equation}\label{aim1}
\begin{split}
\sideset{}{^\prime}\sum_{j=1}^\infty e^{-\frac{(t_j-T)^2}{M^2}}|\nu_j(1)|^2\left|\sum_{p\leq x^3}
\frac{\lambda_j(p)\Lambda_x(p)\log(xp)}{p^{\sigma+it}}\right|^{4n}
&\ll MT\left(\sum_{p\leq x^3}\frac{|\Lambda_x(p)\log(xp)p^{1/2-\sigma}|^2}{p}\right)^{2n}\\
&\ll MT\left((\log x)^2\sum_{p\leq x^3}\frac{(\log p)^2}{p}\right)^{2n}
\ll MT(\log x)^{8n}.
\end{split}
\end{equation}
Thus,
\begin{equation*}
\begin{split}
&\sideset{}{^\prime}\sum_{j=1}^\infty e^{-\frac{(t_j-T)^2}{M^2}}|\nu_j(1)|^2\left(\int_{1/2}^\infty
x^{1/2-\sigma}\left|\sum_{p\leq x^3}
\frac{\lambda_j(p)\Lambda_x(p)\log(xp)}{p^{\sigma+it}}\right|\mathrm{d}\sigma\right)^{4n}\\
\ll &\frac{1}{(\log x)^{4n-1}}\int_{1/2}^\infty
x^{1/2-\sigma}MT(\log x)^{8n}\mathrm{d}\sigma\ll MT(\log x)^{4n}.
\end{split}
\end{equation*}
Inserting this bound to \eqref{after Cauchy} and using Lemma \ref{lemma 6}, we have the third term in \eqref{R}
is of $O(MT)$.
Using the same or even simpler argument as the one dealing with the third term, we get the fourth term in \eqref{R}
in fact is of $o(MT)$. Thus we complete the proof of this lemma.
\end{proof}
\noindent\textbf{Proof of Theorem \ref{main-theorem}.}
By the binomial theorem, we have
\begin{equation*}
\begin{split}
S_j(t)^n=M_j(t)^n+O_n\left(\sum_{\ell=1}^\infty|M_j(t)|^{n-\ell}|R_j(t)|^{\ell}\right).
\end{split}
\end{equation*}
For $1\leq\ell\leq n$, we apply the generalized H\"{o}lder's inequality and
Lemma \ref{lemma 6 spectral}, and get
\begin{equation*}
\begin{split}
&\sideset{}{^\prime}\sum_{j=1}^\infty e^{-\frac{(t_j-T)^2}{M^2}}|\nu_j(1)|^2|M_j(t)|^{n-\ell}|R_j(t)|^{\ell}\\
\ll &\left(\sideset{}{^\prime}\sum_{j=1}^\infty e^{-\frac{(t_j-T)^2}{M^2}}|\nu_j(1)|^2\right)^{\frac{1}{2}}
\left(\sideset{}{^\prime}\sum_{j=1}^\infty e^{-\frac{(t_j-T)^2}{M^2}}|\nu_j(1)|^2 |M_j(t)|^{2n}\right)^{\frac{n-\ell}{2n}}\\
&\times\left(\sideset{}{^\prime}\sum_{j=1}^\infty e^{-\frac{(t_j-T)^2}{M^2}}|\nu_j(1)|^2 |R_j(t)|^{2n}\right)^{\frac{\ell}{2n}}\\
\ll&_{t,n}TM(\log\log T)^{(n-\ell)/2}\ll_{t,n}TM(\log\log T)^{(n-1)/2}.
\end{split}
\end{equation*}
For the second argument of theorem,
note that the $n$-th moment of $\mu_{TM}$ is
\bna
\int_{\mathbb{R}}\xi^n\mathrm{d}\mu_{TM}(\xi)=\left(\sideset{}{^\prime}\sum_{|t_j-T|\leq M}|\nu_j(1)|^2
\left(\frac{S_j(t)}{\sqrt{\log\log T}}\right)^n\right)
\bigg/\left(\sideset{}{^\prime}\sum_{|t_j-T|\leq M}|\nu_j(1)|^2\right).
\ena
Using Theorem \ref{main-theorem} and Lemma \ref{corollary1}, we obtain that for all $n\geq 0$,
\bna
\lim_{T\rightarrow \infty}\int_{\mathbb{R}}\xi^n\mathrm{d}\mu_{TM}(\xi)=C_n=
\sqrt{\pi}\int_{\mathbb{R}}\xi^n \exp(-\pi^2\xi^2)\mathrm{d}\xi,
\ena
where we used the theory of moments in probability theory (see, for example, \cite[Lemma 30.2]{Billingsley}).

\section{Proof of Lemma \ref{corollary1}}\label{Prove Lemma}
The strategy of this proof is similar to \cite[Lemma 1]{HL}.
By \eqref{lambda_j(n) bound} and \eqref{classic bounds},
\begin{equation*}
\begin{split}
\sideset{}{^\prime}\sum_{j} e^{-\frac{(t_j-T)^2}{M^2}}|\nu_j(1)|^2\lambda_j(m)\lambda_j(n)
=\sideset{}{^\prime}\sum_{j}& h(t_j)|\nu_j(1)|^2\lambda_j(m)\lambda_j(n)\\
&+O\left(e^{-\frac{T^2}{M^2}}\tau(m)\tau(n)(mn)^\vartheta M^2\right),
\end{split}
\end{equation*}
where $h(t)$ is defined by \eqref{definition h(t)}.
Applying the Kuznetsov trace formula to the right-hand side of above and \eqref{bounds},
we have
\bna
\int_{-\infty}^\infty h(t)\tanh(\pi t)t\, \mathrm{d}t=2\sqrt{\pi}TM+O(1),
\ena
and
\bna
\int_{-\infty}^\infty \frac{\tau_{it}(m)\tau_{it}(n)}{|\zeta(1+2it)|^2}h(t)\mathrm{d}t=O\big(\tau(m)\tau(n)M(\log T)^3\big).
\ena
So it suffices to show that
\bea\label{aim 1}
\mathfrak{J}:=\sum_{c>0}
\frac{S(m,n;c)}{c}\int_{-\infty}^\infty J_{2it}\left(\frac{4\pi\sqrt{mn}}{c}\right)\frac{h(t)t}{\cosh(\pi t)}\,\mathrm{d}t
=O(1/\sqrt{T}),
\eea
and
\bea\label{aim 2}
\mathfrak{K}:=\sum_{c>0}
\frac{S(-m,n;c)}{c}\int_{-\infty}^\infty K_{2it}\left(\frac{4\pi\sqrt{mn}}{c}\right)\sinh(\pi t)h(t)t\mathrm{d}t=O(MT^\varepsilon).
\eea
Let $L$ be any integer greater than $1/\varepsilon$, we shift the $t$-integral in \eqref{aim 1} to $\Im(t)=-L$ with simple poles $t=i(\frac{1}{2}-\ell)$ for integers $\ell$, and obtain
that each $t$-integral is
\bea\label{t-integral}
\sum_{\ell=1}^Lc_{\ell}J_{2\ell-1}\left(\frac{4\pi\sqrt{mn}}{c}\right)
h\big(i\big(1/2-\ell\big)\big)
+\int_{-\infty}^\infty J_{2L+2it}\left(\frac{4\pi\sqrt{mn}}{c}\right)
\frac{h(t-iL)(t-iL)}{\cosh(\pi (t-iL))}\,\mathrm{d}t,
\eea
here $c_{\ell}=i(-1)^{\ell}(1-2\ell)$.
Combining with \eqref{J-Bessel} and the bounds
\bna
h(i(1/2-\ell))
\ll e^{-\frac{T^2}{M^2}}e^{\frac{L^2}{M^2}},\quad
h(t-iL)\ll h(t),\quad
|\cosh(\pi (t-iL))|=\cosh(\pi t)\gg e^{\pi |t|},
\ena
we get that \eqref{t-integral} is dominated by
\bna
O_L\left(MT\log T\left(\frac{\sqrt{mn}}{cT}\right)^{2L}
+\frac{1}{c}e^{-\frac{T^2}{2M^2}}\right).
\ena
Using Weil's bound, $\sqrt{mn}\leq T^{1-\varepsilon}$ and $(m,n,c)\leq T$,
we obtain $\mathfrak{J}=O(1/\sqrt{T})$.

Next we denote $Y=\frac{4\pi\sqrt{mn}}{c}$ and split $\mathfrak{K}$ in \eqref{aim 2} into two parts denoted by
$\mathfrak{K}=\mathfrak{K}^++\mathfrak{K}^-,$
which depends on the magnitude of $c$\,:
$c\geq \sqrt{mn}$ ($Y\leq 4\pi$) and $c<\sqrt{mn}$ ($Y> 4\pi$).
\subsection{Case $\mathfrak{K}^+$}
Recall that
\begin{equation*}
\begin{split}
\mathfrak{K}^+=\sum_{c\geq \sqrt{mn}}\frac{S(-m,n;c)}{c}
\int_{-\infty}^{+\infty}h(t)t\sinh(\pi t)K_{2it}\left(Y\right)\mathrm{d}t.
\end{split}
\end{equation*}
By the identity $\sin(2iz)=2i\sinh (z)\cosh (z)$ and \eqref{K-Bessel}, we can obtain
\bna
\int_{-\infty}^{+\infty}h(t)t\sinh(\pi t)K_{2it}\left(Y\right)\mathrm{d}t
=\frac{\pi i}{2}\int_{-\infty}^\infty \frac{I_{2it}(Y)}{\cosh(\pi t)}h(t)t\mathrm{d}t.
\ena
Since $Y\leq 4\pi$,
in this case the treatment is exactly the same as $\mathfrak{J}$ by using \eqref{I-Bessel},
we can get the contribution of $\mathfrak{K}^+$ satisfies the bound \eqref{aim 2}.
\subsection{Case $\mathfrak{K}^-$}
By splitting the $t$-integral of $\mathfrak{K}^-$ into $\int_0^{T-M\log T}$,
$\int_{T-M\log T}^{T+M\log T}$ and $\int_{T+M\log T}^\infty$,
we denote
\bna
t\text{-integral}=2\int_0^\infty h(t)t\sinh(\pi t)K_{2it}\left(Y\right)\mathrm{d}t:=\mathfrak{K}_1^-(Y)+\mathfrak{K}_2^-(Y)+\mathfrak{K}_3^-(Y).
\ena
With the help of the bound \eqref{K-Bessel bound},
we obtain
\bna
\mathfrak{K}_1^-(Y)+\mathfrak{K}_3^-(Y)\ll T^{2/3-\log T}.
\ena
Thus the contribution of $\mathfrak{K}_i^-(Y)$ ($i=1,3$) to $\mathfrak{K}$ is negligible as $T\rightarrow \infty$.

Recall $t\in [T-M\log T, T+M\log T]$ in $\mathfrak{K}_2^-(Y)$ and $\sqrt{mn}\leq T^{1-\varepsilon}$,
so that $Y=o(t)$ in $\mathfrak{K}_2^-(Y)$.
Note that the asymptotic formula (see \cite[pg. 142]{MOS})
\bna
K_{2it}(Y)=\frac{\sqrt{2\pi}e^{-\pi t}}{(4t^2-Y^2)^{\frac{1}{4}}}\big\{\sin (\phi_{Y}(t))+O(t^{-1})\big\},
\ena
where
\bna
\phi_{Y}(t)=\frac{\pi}{4}+2t\arccosh\left(\frac{2t}{Y}\right)-\sqrt{4t^2-Y^2}.
\ena
As Liu \cite{Liu} points out, the error term $O\left(Y^{-1}\right)$ in \cite[pg. 142]{MOS} can be replaced here by $O(t^{-1})$,
by the power series expansion for $K_{2it}(x)$ through that of $I_{\pm2it}(x)$.
For $t\in [T-M\log T, T+M\log T]$, we have
\bna
\sinh(\pi t)e^{-\pi t}=\frac{1}{2}+O(e^{-\pi T}),
\ena
and
\bna
\frac{1}{(4t^2-Y^2)^{\frac{1}{4}}}=\frac{1}{\sqrt{2t}}+O(Y^2t^{-5/2}).
\ena
We only need to deal with
\bna
\int_{T-M\log T}^{T+M\log T} e^{-\left(\frac{t-T}{M}\right)^2}\sin (\phi_Y(t))t^{1/2}\mathrm{d}t
=\int_{T-M\log T}^{T+M\log T} e^{-\left(\frac{t-T}{M}\right)^2}\frac{e^{i\phi_Y(t)}-e^{-i\phi_Y(t)}}{2i}t^{1/2}\mathrm{d}t,
\ena
because obviously other terms can get some smaller contribution to $\mathfrak{K}$.
Denote
\bna
f(t)=e^{-\left(\frac{t-T}{M}\right)^2}t^{1/2},\quad \quad g(t)=\pm\phi_Y(t),
\ena
then
\bna
g'(t)=\pm2\log\left(\frac{2t+\sqrt{4t^2+Y^2}}{Y}\right)\asymp \log T,
\ena
and
\bna
g''(t)=\pm4(4t^2-Y^2)^{-1/2}\asymp T^{-1}.
\ena
By Fa\`{a} di Bruno's formula for high derivatives of composite functions (see e.g. \cite{Johnson})
\begin{equation*}
\begin{split}
\frac{\mathrm{d}^n}{\mathrm{d}t^n}\left(e^{-\left(\frac{t-T}{M}\right)^2}\right)
=e^{-\left(\frac{t-T}{M}\right)^2}\sum_{m_1+2m_2=n}\frac{(-1)^{m_1+m_2}}{m_1!m_2!}
\left(\frac{2(t-T)}{M}\right)^{m_1}\frac{n!}{M^n}\ll M^{-n}(\log T)^n,
\end{split}
\end{equation*}
so that
\bna
f^{(n)}(t)=\sum_{k=0}^n\binom{n}{k}(t^{1/2})^{(n-k)}
\frac{\mathrm{d}^k}{\mathrm{d}t^k}\left(e^{-\left(\frac{t-T}{M}\right)^2}\right)\ll T^{1/2}M^{-n}(\log T)^n.
\ena
Repeated integration by parts gives
\begin{equation*}
\begin{split}
I(Y,T)&:=\int_{T-M\log T}^{T+M\log T}f(t)e^{ig(t)}\mathrm{d}t
=i\int_{T-M\log T}^{T+M\log T}e^{ig(t)}\left(\frac{f'(t)}{g'(t)}-\frac{f(t)g''(t)}{(g'(t))^2}\right)\mathrm{d}t
+O\big(T^{1/2-\log T}\big)\\
&=i\int_{T-M\log T}^{T+M\log T}e^{ig(t)}\frac{f'(t)}{g'(t)}\mathrm{d}t
+O\big(MT^{-1/2}(\log T)^{-1}\big)\\
&=\int_{T-M\log T}^{T+M\log T}e^{ig(t)}\left(\frac{f''(t)}{(g'(t))^2}-\frac{2f'(t)g''(t)}{(g'(t))^3}\right)\mathrm{d}t
+O\big(MT^{-1/2}(\log T)^{-1}\big)\\
&=\int_{T-M\log T}^{T+M\log T}e^{ig(t)}\frac{f''(t)}{(g'(t))^2}\mathrm{d}t
+O\big(MT^{-1/2}(\log T)^{-1}\big)\\
&\cdots\\
&=c_n\int_{T-M\log T}^{T+M\log T}e^{ig(t)}\frac{f^{(n)}(t)}{(g'(t))^n}\mathrm{d}t
+O_n\big(MT^{-1/2}(\log T)^{-1}\big),
\end{split}
\end{equation*}
where $|c_n|=1$. By the estimates for $f^{(n)}(t)$ and $g'(t)$, and the assumption $M>T^\varepsilon$,
we take sufficiently large $n$ to obtain
\begin{equation*}
\begin{split}
I(Y,T)\ll MT^{-1/2}(\log T)^{-1}.
\end{split}
\end{equation*}
Thus the contribution to $\mathfrak{K}$ through $\mathfrak{K}_2^-$ is
\begin{equation*}
\begin{split}
&\ll MT^{-1/2}(\log T)^{-1}\sum_{c\leq\sqrt{mn}}\frac{(-m,n,c)^{1/2}c^{1/2}\tau(c)}{c}
\ll MT^{-1/2}(\log T)^{-1}\sum_{c\leq\sqrt{mn}}\frac{(n,c)^{1/2}\tau(c)}{c^{1/2}}\\
&\ll MT^{-1/2}(\log T)^{-1}\sum_{d\leq\sqrt{mn}\atop d|n}\tau(d)\sum_{c\leq\sqrt{mn}/d}\frac{\tau(c)}{c^{1/2}}\ll MT^\varepsilon.
\end{split}
\end{equation*}
Thus the proof of Lemma \ref{corollary1} is complete.

\bigskip
\noindent{\bf Acknowledgements}
H. Wang would like to thank the Alfr\'{e}d R\'{e}nyi Institute of Mathematics for providing a great working environment
and the China Scholarship Council for its financial support.

\bibliographystyle{amsplain}

\end{document}